\documentclass[12pt]{article}
\usepackage{amssymb,amsfonts,amsmath,latexsym
,amsthm,graphicx
}
\usepackage[T2A]{fontenc}
\usepackage[cp1251]{inputenc}
\usepackage[ukraineb,russian]{babel}
\usepackage[active]{srcltx}

\begin{document}
\sloppy

\begin{center}
{\large Alexander Nakonechny,
Yuri Podlipenko }
\end{center}

\begin{center}
{\bf \large The minimax approach to the estimation of solutions to 
first order linear systems of ordinary differential periodic equations with inexact data.
 }
\end{center}
\bigskip

\begin{center}
{\bf Abstract}
\end{center}
{\it We consider
first-order linear systems of ordinary differential equations with periodic coefficients. Supposing that right-hand sides of
equations  are not known and subjected to some quadratic restrictions,
we obtain optimal, in certain sense, estimates of solutions to above-mentioned problems from indirect  noisy observations of these solutions on a finite system of points and intervals.
   }

\newtheorem{predlllr}{Remark}
\newtheorem{predlll}{Corollary}
\newtheorem{pred}{Theorem}
\newtheorem{predl}{Lemma}
\newtheorem{predllll}{Proposition}
\newtheorem{predll}{Proposition}
\newtheorem*{predd}{Teoрема}
\newtheorem{predlllv}{Definition}
\newtheorem{pr*}{Theorem 1.}

\newcommand{\n}{^{(1)}}
\newcommand{\m}{^{(2)}}
\newcommand{\p}{^{(1,1)}}
\newcommand{\q}{^{(1,2)}}
\newcommand{\s}{^{(2,1)}}
\newcommand{\w}{^{(2,2)}}

\renewcommand{\proofname}{Proof}

\begin{center}
{\bf \large {Introduction}}
\end{center}

Estimation theory for systems with lumped and distributed parameters under
uncertainty conditions was developed intensively during the last 30 years when essential
results for ordinary and partial differential equations have been
obtained. That
was motivated by the fact that the realistic setting of boundary value problems
describing physical processes often contains perturbations of unknown (or partially
unknown) nature. In such cases the minimax estimation method proved to be
useful, making it possible to obtain optimal estimates both for the unknown
solutions (or right-hand sides of equations appearing in the boundary value
problems) and for linear functionals from them, that is, estimates looked for in
the class of linear estimates with respect to observations
\footnote{
	Here we understand observations of unknown solutions as the functions that are linear transformations of same solutions distorted by additive random noises.},
for which the maximal
mean square error taken over all the realizations of perturbations from certain
given sets takes its minimal value.
Such estimates are called the guaranteed or minimax estimates.

Minimax estimation is studied in a big number of works;
one may refer e.g. to \cite{BIBL5a}--\cite{BIBLkurzh}
and the bibliography therein.

Let us formulate a general approach to the problem. If a state of a system is described by a linear ordinary differential equation
\begin{equation}\label{oopp}
\frac{dx(t)}{dt}=Ax(t)+Bv_1(t),\quad x(t_0)=x_0,
\end{equation}
and a function $y(t)$ is observed in a time interval $[t_0,T]$,
where $y(t)=Hx(t)+v_2(t),$ $x(t) \in \mathbb C^n$, $v_2 \in
\mathbb C^m,$ $y\in \mathbb C^m,$ and $A,\,B,\, H$ are known
matrices, the minimax estimation problem consists in the most
accurate determination of a function $x(t)$ at the "worst"
realization of unknown quantities $(x_0,v_1(\cdot),v_2(\cdot))$
taken from a certain set. N.N. Krasovskii was the first who stated
this problem in \cite{BIBLkras}. Under different constraints imposed
on function $v_2(t)$ and for known function $v_1(t)$ he proposed
various methods of estimating inner products $(a,x(T))$ in the
class of operations linear with respect to observations that
minimize the maximal error. Later these estimates were called
minimax a priori estimates (see \cite{BIBLkras}, \cite{BIBLkurzh}).

Fundamental results concerning estimation under uncertainties were obtained by A. B. Kurzhanskii (see \cite{BIBLkurzh}, \cite{BIBL14}).

 The duality principle elaborated in \cite{BIBLkras},
 \cite{BIBLkurzh}, and \cite{BIBL5a}) proved its efficiency for the determination of minimax estimates \cite{BIBL5a}. According to this principle, finding minimax a priori estimates can be reduced to a certain problem of optimal control of the system adjoint to \eqref{oopp}; this approach enabled one to obtain, under certain restrictions, recurrent equations, namely, the minimax Kalman--Bucy filter (see \cite{BIBL5a}).

The present paper is devoted to the problem of guaranteed estimation for 
systems described by
first-order linear systems of ordinary differential periodic equations with inexact data. From indirect noisy observations of unknown solutions
on a finite system of points and intervals, under quadratic restrictions on unknown right-hand sides of equations, we
find the minimax estimates both for the unknown
solutions
and for linear functionals from them.
It is proved that guaranteed estimates and estimation errors are expressed explicitly from the solutions
of special systems of linear ordinary differential periodic equations, for which the unique solvability
is established.

To do this,
we reduce the guaranteed estimation problem to a certain optimal control problem.
Solving this optimal control problem, we obtain uniquely solvable system of ODEs
via whose solutions the minimax estimates are expressed.

{\bf Preliminaries and auxiliary results}\label{sub1}

Let vector-function $x(t)\in \mathbb C^n$ be a solution of the following  problem
\begin{equation}\label{Cau0p}
\frac{dx(t)}{dt}=A(t)x(t)+B(t)f(t),\quad t\in \mathbb R,
\end{equation}
where $A(t)=[a_{ij}(t)]_{i,j=1}^n$ is an $n\times n$-matrix and
$B(t)=[b_{ij}(t)],i=1,\ldots,n,j=1,\ldots,r,$ is an $n\times r$-matrix with entries $a_{ij}(t)$ and $b_{ij}(t)$ which are continuous $T$-periodic functions,
$f(t)\in \mathbb C^r$ is a $T$-periodic vector-function such that
$f\in (L^2(0,T))^r.$

Denote by $X(t)$ a matrix-valued function $X(t)=[x_1(t),\dots,x_n(t)]$ whose columns are linearly
independent solutions $x_1(t),\dots,x_n(t)$ of the homogeneous system
\begin{equation}\label{homCau0}
\frac{dx(t)}{dt}=A(t)x(t)
\end{equation}
such that $X(0)=E_n,$ where $E_n$ is the unit $n \times n$-matrix. In this case $X(t)$ is said to be
a normalized fundamental matrix of the equation \eqref{homCau0}.

Further we will assume that the
following condition is valid
\begin{equation}\label{detne0}
\det(E_n-X(T))\neq 0.
\end{equation}

Here a solution $x(t)$ of equation \eqref{Cau0p} on the interval $(0,T)$ is interpreted as a continuous solution of the
integral equation
$$
x(t)=x(0)+\int_{0}^t(A(s)x(s)+B(s)f(s))ds
$$
or, equivalently, $x(t)$ satisfies the condition $x(0)=x(T)$ and is absolutely continuous on
$[0,T]$ with its derivative $x'(t)$ satisfying \eqref{Cau0p}
on
$(0,T)$ almost everywhere (except on a set of Lebesgue measure 0).
Outside of the interval $[0,T],$ $x(t)$ is supposed to be extended by periodicity
to the whole real axis, i.e., $x(t+T)=x(t)$ $\forall t\in \mathbb R.$

Under the condition \eqref{detne0} the unique solvability of the problem
\begin{equation}\label{Cau0}
\frac{dx(t)}{dt}=A(t)x(t)+B(t)f(t),\quad t\in (0,T),
\end{equation}
\begin{equation}\label{Cau10}
x(0)=x(T),
\end{equation}
or, what is the same,
the existance of a unique T-periodical solution of equation \eqref{Cau0p}
is established, for example, in \cite{Hart}-\cite{Samoilenko}.

Simultaneously with problem \eqref{Cau0}, \eqref{Cau10}, the following problem:
given vector-function $g(t)\in \mathbb C^n$ such that $g\in (L^2(0,T))^n,$ $r_i\in \mathbb C^N,$ and $t_i\in (0,T),$
$0<t_1<\dots <t_N<T,$ find vector-function $z(t)\in \mathbb C^n$  such that
\begin{equation}\label{adjCau40}
-\frac{dz(t)}{dt}=A^*(t)z(t)+g(t), \quad t\in (0,T),\quad t\neq t_i,
\end{equation}
\begin{equation}\label{adjCau50}
\Delta z\left.\right|_{t=t_i}=z(t_i+0)-z(t_i)=r_i,\quad i=1,\ldots,N,\quad z(0)=z(T),
\end{equation}
is uniquely solvable.\footnote{Here and in what follows, by $\Lambda^*$ we will denote the matrix
complex conjugate and transpose of a matrix
$\Lambda.$}
In fact,
by Theorem 4.1 \cite{Bainov} problem \eqref{adjCau40}--\eqref{adjCau50} will have a unique solution
then and only then
\begin{equation}\label{adjCau40det}
\det(E_n-Z(T))\neq 0,
\end{equation}
where $Z(t)$ is a normalized fundamental matrix of the system
\begin{equation}\label{homadjCau40}
-\frac{dz(t)}{dt}=A^*z(t)
\end{equation}
adjoint to \eqref{homCau0}.
As is known
$Z(t)=[X^*(t)]^{-1}.$
Since $\det(E_n-X(T))\neq 0$ then $E_n-X^*(T)$ is a nonsingular matrix. Multiplying this matrix from the left  by $-[X^*(T)]^{-1},$
we obtain that the matrix
$$
-[X^*(T)]^{-1}(E_n-X^*(T))=E_n-Z(T)
$$
is also nonsingular, i.e., condition \eqref{adjCau40det} is fulfilled.

In addition, from the results containing in chapter 4 of \cite{Bainov} it follows that an a priori estimate
\begin{equation}\label{Cau600}
\|z(t)\|_{\mathbb C^n}\leq K\Bigl\{\int_{0}^T\|g(s)\|_{\mathbb C^n}ds+\sum_{i=1}^N|r_i|\Bigr\}\quad\forall t\in [0,T], \end{equation}
holds,
where $K$ is a constant not depending on $g$ and $r_i$.

Further, the following assertion will be frequently used.
If vector-functions $f(t)\in \mathbb C^n$ and $g (t)\in \mathbb C^n$ are absolutely continuous on the closed interval $[t_1,t_2],$
then the following integration by parts formula is valid
\begin{equation}\label{Cau30}
(f(t_2),g(t_2))_{n}-(f(t_1),g(t_1))_{n}=
\int_{t_1}^{t_2}\Bigl[\Bigl(f(t),\frac{dg(t)}{dt}\Bigr)_{n}+\Bigl(g(t),\frac{df(t)}{dt}\Bigr)_{n}\Bigr]dt,
\end{equation}
where by $(\cdot,\cdot)_{n}$ we denote here and later on the inner product in $\mathbb C^n.$
\begin{predl}\label{lemm1}
 Suppose  $Q$ is a bounded positive\footnote{That  is
$(Qf,f)_H>0$ when $f\neq 0$.} Hermitian (self-adjoint) operator in a complex (real) Hilbert space $H$ with bounded inverse
$Q^{-1}.$
Then, the generalized Cauchy$-$Bunyakovsky inequality
\begin{equation}\label{CBI1}
|(f, g)_H|\leq (Q^{-1}f, f)_H^{1/2}(Qg,g)_H^{1/2}\quad (f,g\in H)
\end{equation}
is valid.
The equality sign in \eqref{CBI1} is attained at the element
$$
g=\frac{Q^{-1}f}{(Q^{-1}f,f)_H^{1/2}}.
$$
\end{predl}
\begin{proof}
Introduce
Hilbert space $\tilde H$ consisting of elements of $H$ endowed with inner product
\begin{equation}\label{CBI11}
(u,v)_{\tilde H}:=(Q^{-1}u,v)_H^{1/2}
\end{equation}
and norm
$$
\|v\|_{\tilde H}:=(Q^{-1}v,v)_H^{1/2}
$$
generated by this inner product.
Due to the properties of the operator $Q,$ the above inner product is well-defined.
Setting in the Cauchy$-$Bunyakovsky inequality
\begin{equation}\label{CBI111}
(u,v)_{\tilde H}\leq \|u\|_{\tilde H}^{1/2}\|v\|_{\tilde H}^{1/2}
\end{equation}
$u=f$ and $v=Qg,$ we obtain
$$
(f,Qg)_{\tilde H}\leq \|f\|_{\tilde H}^{1/2}\|Qg\|_{\tilde H}^{1/2}=(Q^{-1}f,f)_H^{1/2}(Q^{-1}Qg,Qg)_H^{1/2}=(Q^{-1}f, f)_H^{1/2}(Qg,g)_H^{1/2},
$$
where the equality is attained when
$$
g=\frac{Q^{-1}f}{(Q^{-1}f,f)_H^{1/2}}.
$$
\end{proof}

{\bf Problem statement}

In this section we study minimax estimation problems in the case
of point observations and deduce equations generated the minimax estimates of functionals from periodic solutions to the problem
\begin{equation}\label{oopp11}
\frac{dx(t)}{dt}=Ax(t)+B(t)f(t),\quad t\in (0,T),
\end{equation}
\begin{equation}\label{Cau1011}
x(0)=x(T).
\end{equation}
Let $t_i,$ $i=1,\dots,N,$ $0<t_1<\dots<t_N<T$ be a given system of points on the closed
interval $[0,T]$
with $t_{0}=0$ and $t_{N+1}=T$ and let $\Omega_j,$
$j=1,\dots,M,$  be a given system of subintervals of $[0,T]$.
The problem is to estimate the expression
\begin{equation}\label{Cau120vv}
l(x)=\int_{0}^{T}(x(t),l_0(t))_{n}dt,
\end{equation}
from observations of the form
\begin{equation}\label{1p3vvv}
y_i=H_{i}x(t_i)+\xi_i,\quad i=1,\dots,N,
\end{equation}
\begin{equation}\label{sk7s} y_j(t)=H_j(t)x(t)+\xi_j(t),\quad t\in \Omega_j,\quad j=1,\dots,M,
\end{equation}
in the class of estimates
\begin{equation}\label{3p3vvv}
\widehat{l(x)}=\sum_{i=1}^N (y_i,u_i)_m+\sum_{j=1}^M \int_{\Omega_j}(y_j(t),u_j(t))_ldt+c,
\end{equation}
linear with respect to observations (\ref{1p3vvv}), (\ref{sk7s}); here $x(t)$ is
the state of a system described by the Cauchy problem (\ref{Cau0}), (\ref{Cau10}), $l_0\in (L^2(0,T))^n,$
$H_{i}$ are $m\times n$ matrices,
$H_j(t)$ are $l\times n$ matrices with the entries that are continuous functions on $\Omega_j,$ $u_i\in \mathbb C^m,$
$u_j(t)$ are vector-functions belonging to $(L^2(\Omega_j))^l,$
$c\in \mathbb C.$
We suppose that
$f\in G_1,$ where
\begin{equation}\label{Cau90}
G_1=\Big.\Big\{\tilde f\in (L^2(0,T))^r: \int_{0}^{T} (Q(t)(\tilde f(t)-f_0(t)),\tilde f(t)-f_0(t))_r\,dt
\leq 1
\Big.\Big\},
\end{equation}\label{llm}
$\xi:=(\xi_{1},\ldots,\xi_{N},\xi_{1}(\cdot),\ldots,\xi_{M}(\cdot))\in G_2,$ where
$$
\xi_i=
\begin{pmatrix}
\xi_1^{(i)}\\ \vdots\\ \xi_m^{(i)}
\end{pmatrix}
\quad \mbox{and}\quad \xi_j(\cdot)=
\begin{pmatrix}
\xi_1^{(j)}(\cdot)\\ \vdots\\ \xi_l^{(j)}(\cdot)
\end{pmatrix}
$$
are estimation errors in \eqref{1p3vvv} and (\ref{sk7s}), respectively, that are
realizations of random vectors  $\xi_i=\xi_i(\omega)\in \mathbb
C^m$ and random vector-functions $\xi_j(t)=\xi_j(\omega,t)\in \mathbb
C^l$ and $G_2$ denotes the set of random elements
$\tilde \xi=(\tilde \xi_{1},\ldots,\tilde \xi_{N},\tilde \xi_{1}(\cdot),\ldots,\tilde \xi_{M}(\cdot)),$
whose components
$$
\tilde \xi_i=
\begin{pmatrix}
\tilde \xi_1^{(i)}\\ \vdots\\ \tilde \xi_m^{(i)}
\end{pmatrix}
\quad \mbox{and}\quad \tilde \xi_j(\cdot)=
\begin{pmatrix}
\tilde \xi_1^{(j)}(\cdot)\\ \vdots\\ \tilde \xi_l^{(j)}(\cdot)
\end{pmatrix}
$$
are uncorrelated, have zero means, $\mathbb E\tilde \xi_i=0,$ and $\mathbb E\tilde \xi_j(\cdot)=0,$ with
finite
second moments  $\mathbb E |\tilde \xi_i|^2$ and $\mathbb E\|\tilde \xi_j(\cdot)\|_{(L^2(\Omega_j))^l}^2$,  and unknown correlation matrices
$\tilde R_i=\mathbb E\tilde\xi_i\tilde\xi_i^{*}
= [r_{jk}^{(i)}]_{j,k=1}^m$ with entries $r_{jk}^{(i)}=\mathbb E\tilde\xi_j^{(i)}\bar{\tilde\xi}_k^{(i)}$ and unknown correlation matrices $\tilde R_j(t,s)=\mathbb E\tilde\xi_j(t)\tilde\xi_j^{*}(s)$ satisfying the
 conditions
\begin{equation}\label{d6llvvv}
\sum_{i=1}^NSp\,[D_i\tilde R_i]\leq
1,
\end{equation}
and
\begin{equation}\label{d6llvvv'}
\sum_{j=1}^M\int_{\Omega_j} Sp\,[D_j(t)\tilde R_j(t,t)]dt\leq
1.
\end{equation}
correspondingly,
where $D_i=[d_{jk}^{(i)}]_{j,k=1}^m$ and $D_j(t)$ are Hermitian positive definite $m\times m$ and $l\times l$-matrices, respectively
\footnote{$\mbox{Sp}\,D := \sum_{i=1}^ld_{ii}$ denotes the trace of the matrix $D=\{d_{ij}\}_{i,j=1}^l$}.
Here in \eqref{Cau90}, $f_0\in (L^2(0,T))^r$ is a prescribed vector,\label{tt}
$Q(t)$
is a Hermitian positive definite matrix.

Set $u:=(u_1,\ldots,u_N,u_1(\cdot),\ldots,u_M(\cdot))\in \mathbb C^{N\times m}\times (L^2(\Omega_1))^l\times\dots\times (L^2(\Omega_M))^l =:H.$ Norm in space $H$ is defined by
$$
\|u\|_H=\Bigl\{\sum_{i=1}^N\|u_i\|_{\mathbb C^m}^2+\sum_{j=1}^M\|u_j(\cdot)\|_{(L^2(\Omega_j))^l}\Bigr\}^{1/2}.
$$

\begin{predlllv}The estimate
\begin{equation}\label{d6llvvves}
\widehat{\widehat
{l(x)}}=\sum_{i=1}^N(y_i,\hat u_i
)_m+\sum_{j=1}^M \int_{\Omega_j}(y_j(t),\hat u_j(t))_ldt+\hat c,
\end{equation}
in which elements $\hat u_i,$ $\hat u_j(\cdot),$ and a number $\hat c$
are determined from the condition
\begin{equation} \label{11fvvv}
\inf_{u\in
H,c\in \mathbb C}\sigma(u,c)=\sigma(\hat u,\hat c),
\end{equation}
where
$$
\sigma(u,c)=\sup_{\tilde f \in G_1,\, \tilde \xi
\in G_2}\mathbb E|l(\tilde x)-\widehat{l(\tilde x)}|^2,
$$
\begin{equation} \label{llxfvvv}
\widehat{l(\tilde x)}=\sum_{i=1}^N(\tilde y_i,u_i
)_m +\sum_{j=1}^M \int_{\Omega_j}(\tilde y_j(t),u_j(t))_ldt+c,
\end{equation}
\begin{equation} \label{4iufvvv}
\tilde y_{i}=H_{i}\tilde x(t_i)+\tilde \xi_i,\quad i=1,\dots,N, \quad \tilde y_{j}(t)=H_{j}(t)\tilde x(t)+\tilde \xi_j(t),\quad j=1,\dots,M,
\end{equation}
and $\tilde x(t)$ is the solution to the problem (\ref{Cau0}), (\ref{Cau10}) at $f(t)=\tilde f(t),$
will be called the minimax estimate of expression (\ref{Cau120vv}).

The quantity
\begin{equation} \label{12dfv}
\sigma:=\{\sigma(\hat u,\hat c)\}^{1/2}
\end{equation}
will be called the error of the minimax estimation of $l(x).$
\end{predlllv}

{\bf Main results}

For any fixed $u=(u_1,\ldots,u_N,u_1(\cdot),\ldots,u_M(\cdot))\in \mathbb C^{N\times m}\times (L^2(\Omega_1))^l\times\dots\times (L^2(\Omega_M))^l =H$ introduce the vector-function $z(t;u)$
as a unique solution to
the problem\footnote{Here and in what follows we assume that if a function is piecewise continuous then it is continuous from the left.}
\begin{equation}\label{Cau140aavvo}
-\frac{dz(t;u)}{dt}=A^*(t)z(t;u)+l_0(t)-\sum_{j=1}^M \chi_{\Omega_j}(t)H^*_j(t)u_j(t),\quad t\in (0,T),\quad t\neq t_i,
\end{equation}
\begin{equation}\label{Cau160avv'o}
\Delta z(\cdot;u)\left.\right|_{t=t_i}=z(t_i+0;u)-z(t_i;u)=H^*_iu_i,\quad i=1,\ldots,N,\quad z(T;u)=z(0;u),
\end{equation}
where $\chi_{\Omega}(t)$ is a characteristic function of the set $\Omega.$

\begin{predl}
Finding the minimax estimate of functional
$l(x)$ is equivalent to the problem of optimal control of the system (\ref{point1})
with the cost function
\begin{equation}\label{N4llvvv}
I(u)=\int_{0}^{T}(\tilde Q(t)z(t;u),z(t;u))_ndt
+\sum_{i=1}^{N}(D_i^{-1}u_i,u_i)_m+\sum_{j=1}^M\int_{\Omega_j}(D_j^{-1}(t)u_j(t),u_j(t))_ldt
\to \inf_{u\in H},
\end{equation}
where $\tilde Q(t)=B(t)Q^{-1}(t)B^*(t)$.
\end{predl}
\begin{proof}
Denote by $z_{i}(t;u)$ restriction of function $z(t;u)$ to a subinterval $(t_{i-1},t_i)$ of the interval
$(0,T)$ and extend it from this subinterval to the ends $t_{i-1}$ and $t_{i}$ by continuity. Then
$$
L^* z_1(t;u)=l_0(t)-\sum_{j=1}^M \chi_{\Omega_j}(t)H^*_j(t)u_j(t),\quad 0=:t_0<t<t_1, \quad z_2(t_1;u)=z_1(t_1;u)+H^*_1u_1,
$$
$$
L^* z_2(t;u)=l_0(t)-\sum_{j=1}^M \chi_{\Omega_j}(t)H^*_j(t)u_j(t),\quad t_1<t<t_2, \quad z_3(t_2;u)=z_2(t_2;u)+H^*_2u_2,
$$
$$
\ldots\ldots\ldots\ldots\ldots\ldots\ldots\ldots\ldots
\ldots\ldots\ldots\ldots\ldots\ldots\ldots\ldots
$$
\begin{equation}\label{point1}
L^* z_{i}(t;u)=l_0(t)-\sum_{j=1}^M \chi_{\Omega_j}(t)H^*_j(t)u_j(t),\quad t_{i-1}<t<t_{i}, \quad z_{i+1}(t_{i};u)=z_{i}(t_{i};u)+H^*_iu_{i},
\end{equation}
$$
\ldots\ldots\ldots\ldots\ldots\ldots\ldots\ldots\ldots
\ldots\ldots\ldots\ldots\ldots\ldots\ldots\ldots\\
$$
$$
L^* z_{N}(t;u)=l_0(t)-\sum_{j=1}^M \chi_{\Omega_j}(t)H^*_j(t)u_j(t),\quad t_{N-1}<t<t_{N}, \quad z_{N+1}(t_{N};u)=z_{N}(t_{N};u)+H_{N}^*u_{N},
$$
\begin{multline*}
L^* z_{N+1}(t;u)=L^* z_{N+1}(t)=l_0(t)\\-\sum_{j=1}^M \chi_{\Omega_j}(t)H^*_j(t)u_j(t),\quad t_{N}<t<t_{N+1}:=T, \quad z_{N+1}(T;u)=z_{1}(0;u),
\end{multline*}
where\label{same}
$$
z_{i+1}(t_{i};u):=z_{i+1}(t_{i}+0;u),\quad z_{i}(t_{i};u):=z_{i}(t_{i}-0;u),\quad i=1,\dots,N,
$$
$$
L^*z_i(t;u):=-\frac{dz_i(t;u)}{dt}-A^*(t)z_i(t;u),\quad t_{i-1}<t<t_{i},\quad i=1,\dots,N+1.
$$
Obviously,
$$
z(t;u)=\left\{
\begin{array}{lc}
z_1(t;u),&0=t_0<t\leq t_1;\\ \cdots\cdots\cdots&\cdots\cdots\cdots\cdots\\
z_{i}(t;u),&t_{i-1}<t\leq t_{i};\\
\cdots\cdots\cdots&\cdots\cdots\cdots\cdots\\
z_{N+1}(t),&t_{N}<t\leq t_{N+1}=T, \quad z(0)=z(T).
\end{array}
\right.
$$

  Let $\tilde x$ be a solution to problem (\ref{Cau0}), (\ref{Cau10}) at $f(t)=\tilde f(t).$
 From \eqref{Cau120vv} with $x=\tilde x,$ \eqref{4iufvvv},
and the integration by parts formula \eqref{Cau30} with $f(t)=\tilde x(t),$ $g(t)=z(t;u),$
we obtain
$$
l(\tilde x)-\widehat{l(\tilde x)}=\sum_{i=1}^{N+1}\int_{t_{i-1}}^{t_{i}}(\tilde x(t),l_0(t))_{n}dt
-\sum_{i=1}^{N}(\tilde y_i,u_i)_m-\sum_{j=1}^M \int_{\Omega_j}(\tilde y_j(t),u_j(t))_ldt-c
$$
$$
=\sum_{i=1}^{N+1}\int_{t_{i-1}}^{t_{i}}(\tilde x(t),l_0(t)-\sum_{j=1}^M \chi_{\Omega_j}(t)H^*_j(t)u_j(t))_{n}dt
-\sum_{i=1}^{N}(\tilde x(t_i),H_i^*u_i)_n-\sum_{i=1}^{N}(\tilde \xi_i,u_i)_m
$$
$$
-\sum_{j=1}^M \int_{\Omega_j}(\tilde\xi_j(t),u_j(t))_ldt-c
$$
$$
=\sum_{i=1}^{N+1}\int_{t_{i-1}}^{t_{i}}\Bigl(\tilde x(t),-\frac{dz_i(t;u)}{dt}-A^*(t)z_i(t;u)\Bigr)_{n}dt
-\sum_{i=1}^{N}(\tilde x(t_i),H_i^*u_i)_n-\sum_{i=1}^{N}(\tilde \xi_i,u_i)_m
$$
$$
-\sum_{j=1}^M \int_{\Omega_j}(\tilde\xi_j(t),u_j(t))_ldt-c
$$
$$
=\sum_{i=1}^{N+1}\Bigl((\tilde x(t_{i-1}),z_i(t_{i-1};u))_n-(\tilde x(t_{i}),z_i(t_{i};u))_n \Bigr)
+\sum_{i=1}^{N+1}\int_{t_{i-1}}^{t_{i}}\Bigl(\frac{d\tilde x(t)}{dt}-A(t)\tilde x(t),z_i(t;u)\Bigr)_{n}dt
$$
$$
-\sum_{i=1}^{N}(\tilde x(t_i),z_{i+1}(t_i;u)-z_{i}(t_i;u))_n-\sum_{i=1}^{N}(\tilde \xi_i,u_i)_m
-\sum_{j=1}^M \int_{\Omega_j}(\tilde\xi_j(t),u_j(t))_ldt-c
=(x(t_0),z_1(t_0;u))_n
$$
$$
-(z_1(x(t_1),t_1;u))_n+\sum_{i=2}^{N}\Bigl(\tilde x(t_{i-1}),z_i(t_{i-1};u))_n-(\tilde x(t_{i}),z_i(t_{i};u))_n \Bigr)+(\tilde x(t_N),z_{N+1}(t_N))_n
$$
$$
-(\tilde x(t_{N+1}),z_{N+1}(t_{N+1}))_n+\sum_{i=1}^{N+1}\int_{t_{i-1}}^{t_{i}}\Bigl(B(t)\tilde f(t),z_i(t;u)\Bigr)_{n}dt
$$
$$
-\sum_{i=1}^{N}(\tilde x(t_i),z_{i+1}(t_i;u)-z_{i}(t_i;u))_n-\sum_{i=1}^{N}(\tilde \xi_i,u_i)_m
-\sum_{j=1}^M \int_{\Omega_j}(\tilde\xi_j(t),u_j(t))_ldt-c.
$$
Taking into account that
$$
\sum_{i=2}^{N}(\tilde x(t_{i-1}),z_i(t_{i-1};u))_n+(\tilde x(t_N),z_{N+1}(t_N))_n=\sum_{i'=1}^{N-1}(\tilde x(t_{i'}),z_{i'+1}(t_{i'};u))_n
+(\tilde x(t_N),z_{N+1}(t_N))_n
$$
$$
=\sum_{i=1}^{N}(\tilde x(t_{i}),z_{i+1}(t_{i};u))_n,
$$
from latter equalities, we have
$$
l(\tilde x)-\widehat{l(\tilde x)}=\sum_{i=1}^{N+1}\int_{t_{i-1}}^{t_{i}}\Bigl(B(t)\tilde f(t),z_i(t;u)\Bigr)_{n}dt-\sum_{i=1}^{N}(\tilde \xi_i,u_i)_m
-\sum_{j=1}^M \int_{\Omega_j}(\tilde\xi_j(t),u_j(t))_ldt-c
$$
\begin{equation}\label{hhll}
=\int_{0}^{T}\Bigl(B(t)\tilde f(t),z(t;u)\Bigr)_{n}dt-\sum_{i=1}^{N}(\tilde \xi_i,u_i)_m
-\sum_{j=1}^M \int_{\Omega_j}(\tilde\xi_j(t),u_j(t))_ldt-c.
\end{equation}
The latter equality yields
\begin{equation}\label{Cau210}
\mathbb E[l(\tilde x)-\widehat{l(\tilde x)}]=
\int_{0}^{T}\Bigl(\tilde f(t),B^*(t)z(t;u)\Bigr)_{r}dt-c.
\end{equation}
Taking into consideration the known relationship
\begin{equation}\label{disp}
\mathbb {D}\eta=\mathbb E|\eta|^2-|\mathbb E\eta|^2
\end{equation}
that couples the variance $\mathbb {D}\eta=\mathbb
E|\eta-\mathbb E\eta|^2$ of random variable $\eta$ with its expectation
$\mathbb E\eta,$ in which $\eta$ is determined by right-hand side of \eqref{hhll} and noncorrelatedness of $\tilde \xi_i=(\tilde \xi_1^{(i)},\dots,\tilde \xi_m^{(i)})^T$ and $\tilde \xi_j(\cdot)=(\tilde \xi_1^{(j)}(\cdot),
\dots,\tilde \xi_l^{(j)}(\cdot))^T$, from the equalities \eqref{hhll} and \eqref{Cau210} we find
$$
\mathbb E|l(\tilde x)-\widehat{l(\tilde x)}|^2=\Bigl|\int_{0}^{T}\Bigl(\tilde f(t),B^*(t)z(t;u)\Bigr)_{r}dt-c\Bigr|^2
+\mathbb E\Bigl|\sum_{i=1}^{N}(\tilde \xi_i,u_i)_m
+\sum_{j=1}^M \int_{\Omega_j}(\tilde\xi_j(t),u_j(t))_ldt\Bigr|^2
$$
$$
=\Bigl|\int_{0}^{T}\Bigl(\tilde f(t)-f_0(t),B^*(t)z(t;u)\Bigr)_{r}dt
+\int_{0}^{T}\Bigl(f_0(t),B^*(t)z(t;u)\Bigr)_{r}dt
-c\Bigr|^2
$$
$$
+\mathbb E\Bigl|\sum_{i=1}^{N}(\tilde \xi_i,u_i)_m\Bigr|^2+\mathbb E\Bigl|\sum_{j=1}^M \int_{\Omega_j}(\tilde\xi_j(t),u_j(t))_ldt\Bigr|^2.
$$
Thus,
$$
\inf_{c \in \mathbb C}\sigma(u,c)= \inf_{c \in \mathbb C} \sup_{\tilde f \in G_1,\, \tilde \xi \in G_2}
\mathbb E|l(\tilde x)-\widehat{l(\tilde x)}|^2=
$$
$$
=\inf_{c \in \mathbb C}\sup_{\tilde f \in
G_1}\Bigl|\int_{0}^{T}\Bigl(\tilde f(t)-f_0(t),B^*(t)z(t;u)\Bigr)_{r}dt
+\int_{0}^{T}\Bigl(f_0(t),B^*(t)z(t;u)\Bigr)_{r}dt
-c\Bigr|^2
$$
\begin{equation}\label{Cau220}
+\sup_{\tilde \xi \in G_2} \Bigl(\mathbb E\left|\sum_{i=1}^{N}(\tilde \xi_i,u_i)_m\right|^2+\mathbb E\Bigl|\sum_{j=1}^M \int_{\Omega_j}(\tilde\xi_j(t),u_j(t))_ldt\Bigr|^2\Bigr)
\end{equation}
Set
$$
y:=\int_{0}^{T}\Bigl(\tilde f(t)-f_0(t),B^*(t)z(t;u)\Bigr)_{r}dt,
$$
$$
d=c-\int_{0}^{T}\Bigl(f_0(t),B^*(t)z(t;u)\Bigr)_{r}dt.
$$
Then
Lemma 1 and \eqref{Cau90}
imply
$$
|y|\leq \Bigl|\int_{0}^{T}(Q^{-1}(t)B^*(t)z(t;u),B^*(t)z(t;u))_rdt\Bigr|^{1/2}
\Bigl|
\int_{t_0}^{t_1} (Q(t)(\tilde f(t)-f_0(t)),\tilde f(t)-f_0(t))_r\,dt\Bigr|^{1/2}
$$
\begin{equation}\label{Cau230}
\leq \Bigl|\int_{0}^{T}(\tilde Q(t)z(t;u),z(t;u))_ndt\Bigr|^{1/2}=:l.
\end{equation}
The direct substitution shows that last inequality
is transformed to an equality at $f^{(0)}\in G_1,$ where
$$
f^{(0)}=f_0 \pm \frac {Q^{-1}B^*(\cdot)z(\cdot;u)}{(Q^{-1}(\cdot)B^*(\cdot)z(\cdot;u),B^*(\cdot)z(\cdot;u))_{(L^2(t_0,t_1))^r}^{1/2}}.
$$
Taking into account the equality
$$
\inf_{d\in \mathbb C}\sup_{|y|\leq l}|y-d|^2=l^2,
$$
we find
$$
\inf_{c\in \mathbb C}\sup_{\tilde f\in G_1}\Bigl|\int_{0}^{T}\Bigl(\tilde f(t)-f_0(t),B^*(t)z(t;u)\Bigr)_{r}dt
+\int_{0}^{T}\Bigl(f_0(t),B^*(t)z(t;u)\Bigr)_{r}dt-c\Bigr|^2
$$
\begin{equation}\label{Cau240}
=l^2=\int_{0}^{T}(\tilde Q(t)z(t;u),z(t;u))_ndt,
\end{equation}
where the infimum over $c$ is attained at
\begin{equation}\label{infc}
c=\int_{0}^{T}\Bigl(f_0(t),B^*(t)z(t;u)\Bigr)_{r}dt.
\end{equation}
Calculate the last term on the right-hand side of (\ref{Cau220}).
Applying Lemma 1, we have
$$
\mathbb E\left|\sum_{i=1}^{N}(\tilde \xi_i,u_i)_m\right|^2
\leq
\mathbb E\left[\sum_{i=1}^N(D^{-1}_iu_i,u_i)_m
\cdot \sum_{i=1}^N(D_i\tilde \xi_i,\tilde \xi_i)_m\right]
 $$
\begin{equation}\label{Cau250}
=\sum_{i=1}^N(D^{-1}_iu_i,u_i)_m\cdot
\mathbb E
\left[\sum_{i=1}^N(D_i\tilde \xi_i,\tilde \xi_i)_m\right].
\end{equation}
Transform the last factor
on the right-hand side of (\ref{Cau250}):
$$
\mathbb E
\left[\sum_{i=1}^N(D_i\tilde \xi_i,\tilde \xi_i)_m\right]
=\sum_{i=1}^N\mathbb E\left(\sum_{j=1}^m\sum_{k=1}^md_{jk}^{(i)}\tilde \xi_k^{(i)}\bar{\tilde \xi}_j^{(i)}\right)
=\sum_{i=1}^N\sum_{j=1}^m\sum_{k=1}^md_{jk}^{(i)}\mathbb E\tilde \xi_k^{(i)}\bar{\tilde \xi}_j^{(i)}
$$
$$
\sum_{i=1}^N\sum_{j=1}^m\sum_{k=1}^md_{jk}^{(i)}r_{kj}^{(i)}=\sum_{i=1}^NSp\,[D_i\tilde R_i].
$$
Analogously,
$$
\mathbb E\Bigl|\sum_{j=1}^M \int_{\Omega_j}(\tilde\xi_j(t),u_j(t))_ldt\Bigr|^2\leq
\sum_{j=1}^M\int_{\Omega_j}(D^{-1}_j(t)u_j(t),u_j(t))_ldt\cdot
\mathbb E
\left[\sum_{j=1}^M\int_{\Omega_j}(D_j(t)\tilde \xi_j(t),\tilde \xi_j(t))_ldt\right]
$$
and
$$
\mathbb E
\left[\sum_{j=1}^M\int_{\Omega_j}(D_j(t)\tilde \xi_j(t),\tilde \xi_j(t))_ldt\right]=\sum_{j=1}^M\int_{\Omega_j}Sp\,[D_j(t)\tilde R_j(t,t)]dt
$$
Taking into account (\ref{d6llvvv}) and (\ref{d6llvvv'}0 we deduce from  (\ref{Cau250})
$$
\mathbb E\left|\sum_{i=1}^{N}(u_i,\tilde \xi_i)_m\right|^2+\mathbb E\Bigl|\sum_{j=1}^M \int_{\Omega_j}(\tilde\xi_j(t),u_j(t))_ldt\Bigr|^2
\leq\sum_{i=1}^N(D^{-1}_iu_i,u_i)_m+\sum_{j=1}^M\int_{\Omega_j}(D^{-1}_j(t)u_j(t),u_j(t))_ldt
$$
It is not difficult to check that here, the equality sign is attained at the element
$$
\xi^{(0)}=(\xi_{1}^{(0)},\ldots, \xi_{N}^{(0)},\xi_{1}^{(0)}(\cdot),\ldots,\xi_{M}^{(0)}(\cdot))\in G_2
$$
with
$$
\xi_i^{(0)}=
\frac{\eta D_i^{-1}u_i}{
\left[\sum_{i=1}^N\left(D_i^{-1}u_i,u_i\right)_m\right]^{1/2}},\quad i=1,\dots,N,
$$
$$
\xi_j^{(0)}(t)=
\frac{\eta D_j^{-1}(t)u_j(t)}{
\left[\sum_{j=1}^M\int_{\Omega_j}\left(D_j^{-1}(t)u_j(t),u_j(t)\right)_ldt\right]^{1/2}},\quad j=1,\dots,M,
$$
where $\eta$ is a random variable such that $\mathbb E\eta=0$ and
$\mathbb E|\eta|^2=1.$
Hence,
\begin{multline}\label{Cau250'}
\sup_{\tilde \xi \in G_2} \Bigl(\mathbb E\left|\sum_{i=1}^{N}(\tilde \xi_i,u_i)_m\right|^2+\mathbb E\Bigl|\sum_{j=1}^M \int_{\Omega_j}(\tilde\xi_j(t),u_j(t))_ldt\Bigr|^2\Bigr)\\
=\sum_{i=1}^N(D^{-1}_iu_i,u_i)_m+\sum_{j=1}^M\int_{\Omega_j}(D^{-1}_j(t)u_j(t),u_j(t))_ldt.
\end{multline}
The statement of the lemma follows now from (\ref{Cau220}), (\ref{Cau240}), (\ref{infc})  and
(\ref{Cau250'}).
The proof is complete.
\end{proof}

Further in the proof of Theorem \ref{ii} stated below,
it will be shown that
solving the optimal control problem  (\ref{Cau140aavvo})$-$(\ref{N4llvvv}) is reduced to solving some system of differential equations.
\begin{pred}\label{ii}
The minimax estimate $\widehat{\widehat{l(x)}}$ of expression $l(x)$ has the form
$$
\widehat{\widehat{l(x)}}=\sum_{i=1}^N(y_i,\hat u_i)_m+\sum_{j=1}^M \int_{\Omega_j}(y_j(t),\hat u_j(t))_ldt+\hat c=l(\hat x),
$$
where
\begin{multline}\label{ddd7llvvv}
\hat{u}_i=D_iH_ip(t_i),\quad i=1,\ldots,N,\quad \hat u_j(t)=D_j(t)H_j(t)p(t), \\ j=1,\ldots,M,\quad \hat c=\int_{0}^{T}\Bigl(f_0(t),B^*(t)\hat z(t)\Bigr)_{r}dt,
\end{multline}
and vector-functions $p(t),$ $\hat z(t)$, and $\hat x(t)$
are determined from
the solution of the systems of equations
\begin{equation}\label{Cau140aavv'}
-\frac{d\hat z(t)}{dt}=A^*(t)\hat z(t)+l_0(t)-\sum_{j=1}^M \chi_{\Omega_j}(t)H^*_j(t)D_j(t)H_j(t)p(t),\quad t\in (0,T),\quad t\neq t_i,
\end{equation}
\begin{equation}\label{Cau160avv'}
\Delta\hat z\left.\right|_{t=t_i}=\hat z(t_i+0)-\hat z(t_i)=H^*_iD_iH_ip(t_i),\quad i=1,\ldots,N,\quad \hat z(T)=\hat z(0),
\end{equation}
\begin{equation}\label{Cau170avv'}
\frac{dp(t)}{dt}=A(t)p(t)+\tilde Q(t)\hat z(t),\quad t\in (0,T),\quad t\neq t_i,
\end{equation}
\begin{equation}\label{Cau140aavv}
\Delta p\left.\right|_{t=t_i}=p(t_i+0)-p(t_i)=0,\quad i=1,\ldots,N,\quad p(0)=p(T)
\end{equation}
and
\begin{equation}\label{Cau160avv}
-\frac{d\hat p(t)}{dt}=A^*(t)\hat p(t)-\sum_{j=1}^M \chi_{\Omega_j}(t)H^*_j(t)D_j(t)[H_j(t)\hat x(t)-y_j(t)],\quad t\in (0,T),\quad t\neq t_i,
\end{equation}
\begin{equation}\label{Cau170avv}
\Delta\hat p\left.\right|_{t=t_i}=\hat p(t_i+0)-\hat p(t_i)=H^*_iD_i[H_i\hat x(t_i)-y_i],\quad i=1,\ldots,N,\quad \hat p(T)=\hat p(0),
\end{equation}
\begin{equation}\label{Cau180aavv}
\frac{d\hat x(t)}{dt}=A(t)\hat x(t)+\tilde Q(t)\hat p(t)+B(t)f_0(t),\quad t\in (0,T),\quad t\neq t_i,
\end{equation}
\begin{equation}\label{Cau190aavv}
\Delta\hat x\left.\right|_{t=t_i}=\hat x(t_i+0)-\hat x(t_i)=0,\quad i=1,\ldots,N,\quad\hat x(0)=\hat x(T),
\end{equation}
respectively.
Problems
\eqref{Cau140aavv'} -- \eqref{Cau140aavv} and \eqref{Cau160avv} -- \eqref{Cau190aavv} are uniquely solvable.
Equations \eqref{Cau160avv} -- \eqref{Cau190aavv}  are fulfilled with probability $1.$

The minimax estimation error $\sigma$ is determined by the formula
\begin{equation} \label{ddd8vvv}
\sigma=[l(p)]^{1/2}.
\end{equation}
\end{pred}
\begin{proof}
First notice that that functional $I(u)$ can be represented in the form
\begin{multline}\label{N4llvvv'}
I(u)=\int_{0}^{T}(\tilde Q(t)z(t;u),z(t;u))_ndt
+\sum_{i=1}^N(D_i^{-1}u_i,u_i)_m\\+\sum_{j=1}^M\int_{\Omega_j}(D_j^{-1}(t)u_j(t),u_j(t))_ldt=\tilde I(u)+L(u)+C,
\end{multline}
where
\begin{equation*}
\tilde I(u)=\int_{0}^{T}(Q(t)\tilde z(t;u),\tilde z(t;u))_ndt
+\sum_{i=1}^N(D_i^{-1}u_i,u_i)_m+\sum_{j=1}^M\int_{\Omega_j}(D_j^{-1}(t)u_j(t),u_j(t))_ldt,
\end{equation*}
\begin{equation*}
L(u)=\int_{0}^{T}(Q(t)z^0(t),\tilde z(t;u))_ndt,\quad
C=\int_{0}^{T}(Q(t)z^{0}(t),z^{0}(t))_ndt,
\end{equation*}
$\tilde z(t;u)$ is a solution of problem \eqref{point1} at $l_0(t)=0$  and
$z^0(t)$ is a solution of the same problem at $u=0.$

From \eqref{Cau600}
we obtain
\begin{equation*}\label{}
\|\tilde z(\cdot;u)\|_{(L^2(t_0,T))^n}\leq c_1\|u\|_{H},
\end{equation*}
whence,
$$\tilde I(u)\leq c_2\Bigl(\|\tilde z(\cdot;u)\|_{(L^2(t_0,T))^n}^2+\|u\|_{H}^2\Bigr)\leq c_3\|u\|_{H}^2,$$
where $c_1,$ $c_2,$ and $c_3$ are constants.
This inequality means that
quadratic form $\tilde I(u)$ is continuous in the space $H.$
Analogously we can show that $L(u)$ is a linear continuous functional in $H.$

It follows from here that $I(u)$ is a continuous strictly convex functional on $H.$ Then, by Corollary 1.8.3 from \cite{Bala},
$I(u)$ is a weak lower semicontinuous strictly convex functional on $H$. Therefore, since
$$
I(u)\geq \sum_{i=1}^N(D_i^{-1}u_i,u_i)_m +\sum_{j=1}^M\int_{\Omega_j}(D_j^{-1}(t)u_j(t),u_j(t))_ldt\geq
c\|u\|_H^2 \quad
\forall u\in H,\,\, \mbox{c=const},
$$
then, by Theorems 13.2 and 13.4 (see \cite{Badr}), there exists one and only one element $\hat
{u}\in H$ such that
$I(\hat{u})=\inf_{u\in H}I(u)$.
Hence, for any fixed $
v\in H$ and $\tau \in \mathbb R$ the functions $s_1(\tau):=I(\hat{
u}+\tau v)$ and $s_2(\tau):=I(\hat{
u}+i\tau v)$ reach their  minimums at a unique point $\tau
=0,$
 so that
\begin{equation}\label{7.82mm}
\frac{d}{d\tau}I(\hat u+\tau
v)\Bigl.\Bigr|_{\tau=0}=0\quad\mbox{and}\quad
\frac{d}{d\tau}I(\hat u+i\tau v)\Bigl.\Bigr|_{\tau=0}=0,
\end{equation}
where $i=\sqrt{-1}.$
Since $z(t;\hat u + \tau v) = z(t; \hat u) + \tau\tilde{z}(t;v),$
where $\tilde{z}(t;v)$ is a unique solution to problem \eqref{Cau140aavvo}, \eqref{Cau160avv'o} at $l_0=0$ and $u=v,$ from (\ref{N4llvvv}) and
\eqref{7.82mm}, we obtain
\begin{equation*}\label{oollb1x}
0=\mbox{Re}\Bigl\{\sum_{i=1}^{N+1}\int_{t_{i-1}}^{t_{i}}(\tilde Q(t)z_{i}(t;\hat u),\tilde z_{i}(t;v))_ndt
+\sum_{i=1}^N(D_i^{-1}\hat u_i, v_i)_m+\sum_{j=1}^M\int_{\Omega_j}(D_j^{-1}(t)\hat u_j(t),v_j(t))_ldt\Bigr\},
\end{equation*}
and
\begin{equation*}\label{oollb2x}
0=\mbox{Im}\Bigl\{\sum_{i=1}^{N+1}\int_{t_{i-1}}^{t_{i}}(\tilde Q(t)z_{i}(t;\hat u),\tilde z_{i}(t;v))_ndt
+\sum_{i=1}^N(D_i^{-1}\hat u_i, v_i)_m+\sum_{j=1}^M\int_{\Omega_j}(D_j^{-1}(t)\hat u_j(t),v_j(t))_ldt\Bigr\},
\end{equation*}
where  $\tilde z_{i}(t;v)$ have the same sense as $z_{i}(t;u)$, $i=1,\dots, N+1$ (see page \pageref{same}).
Whence,
\begin{equation}\label{oollb}
0=\sum_{i=1}^{N+1}\int_{t_{i-1}}^{t_{i}}(\tilde Q(t)z_{i}(t;\hat u),\tilde z_{i}(t;v))_ndt
+\sum_{i=1}^N(D_i^{-1}\hat u_i, v_i)_m+\sum_{j=1}^M\int_{\Omega_j}(D_j^{-1}(t)\hat u_j(t),v_j(t))_ldt.
\end{equation}
Let $p(t)$ be a solution of the problem
\begin{equation}\label{aCau170avv'}
\frac{dp(t)}{dt}=A(t)p(t)+\tilde Q(t)\hat z(t;\hat u),\quad t\in (0,T),\quad t\neq t_i,
\end{equation}
\begin{equation}\label{aCau140aavv}
\Delta p(t)\left.\right|_{t=t_i}=p(t_i+0)-p(t_i)=0,\quad i=1,\ldots,N,\quad p(0)=p(T)
\end{equation}
and
$p_i(t)$ be restriction of $p(t)$ to $(t_{i-1},t_{i})$ extended by continuity to $t_{i-1}$ and $t_{i},$ $i=1,\dots,N+1,$ satisfying the equations
\begin{gather}
Lp_1(t)=\tilde Q(t)z_1(t;\hat u),\quad 0=t_0<t<t_1, \quad p_1(t_0)=p_{N+1}(t_{N+1}),\notag\\
Lp_2(t)=\tilde Q(t)z_2(t;\hat u),\quad t_1<t<t_2, \quad p_2(t_1)=p_1(t_1),\notag\\
\ldots\ldots\ldots\ldots\ldots\ldots\ldots\ldots\ldots
\ldots\ldots\ldots\ldots\ldots\ldots\ldots\ldots\notag\\
Lp_{i}(t)=\tilde Q(t)z_i(t;\hat u),\quad t_{i-1}<t<t_{i}, \quad p_{i}(t_{i-1})=p_{i-1}(t_{i-1}),\label{point1a}\\
\ldots\ldots\ldots\ldots\ldots\ldots\ldots\ldots\ldots
\ldots\ldots\ldots\ldots\ldots\ldots\ldots\ldots\notag\\
Lp_{N}(t)=\tilde Q(t)z_N(t;\hat u),\quad t_{N-1}<t<t_{N}, \quad p_{N}(t_{N})=p_{N-1}(t_{N}),\notag\\
Lp_{N+1}(t)=\tilde Q(t)z_{N+1}(t),\quad t_{N}<t<t_{N+1}=T, \quad p_{N+1}(t_N)=p_{N}(t_N),\notag
\end{gather}
where
$$
Lp_i(t):=\frac{dp_i(t)}{dt}-A(t)p_i(t),\quad t_{i-1}<t<t_{i},\quad i=1,\dots,N+1.
$$
Then we have
$$
\sum_{i=1}^{N+1}\int_{t_{i-1}}^{t_{i}}(\tilde Q(t)z_{i}(t;\hat u),\tilde z_{i}(t;v))_ndt
=\sum_{i=1}^{N+1}\int_{t_{i-1}}^{t_{i}}\Bigl(\frac{dp_i(t)}{dt}-A(t)p_i(t),\tilde z_{i}(t;v)\Bigr)_ndt
$$
$$
=\sum_{i=1}^{N+1}\Bigl((p_i(t_i),\tilde z_{i}(t_i;v))_n-(p_i(t_{i-1}),\tilde z_{i}(t_{i-1};v))_n \Bigr)
$$
$$
+\sum_{i=1}^{N+1}\int_{t_{i-1}}^{t_{i}}\Bigl(p_{i}(t),-\frac{d\tilde z_i(t;v)}{dt}-A^*(t)\tilde z_i(t;v)\Bigr)_ndt
$$
$$
=\sum_{i=1}^{N}(p_i(t_i),\tilde z_{i}(t_i;v))_n+(p_{N+1}(t_{N+1}),\tilde z_{N+1}(t_{N+1};v))_n-(p_{1}(t_{0}),\tilde z_{1}(t_{0};v))_n
$$
$$
-\sum_{i=2}^{N+1}(p_i(t_{i-1}),\tilde z_{i}(t_{i-1};v))_n=\sum_{i=1}^{N}(p_i(t_i),\tilde z_{i}(t_i;v))_n-\sum_{i=1}^{N}(p_{i+1}(t_{i}),\tilde z_{i+1}(t_{i};v))_n
$$
$$
=-\sum_{j=1}^M \int_{\Omega_j}(p(t),H^*_j(t)v_j(t))_ndt+\sum_{i=1}^{N}(p_i(t_i),\tilde z_{i}(t_i;v)-\tilde z_{i+1}(t_i;v))_n
$$
\begin{equation}\label{oollb1}
=-\sum_{j=1}^M \int_{\Omega_j}(p(t),H^*_j(t)v_j(t))_ndt-\sum_{i=1}^{N}(p_i(t_i),H_i^*v_i)_n.
\end{equation}
From \eqref{oollb} and \eqref{oollb1}, we find
\begin{equation}\label{oollb1x}
\hat{u}_i=D_iH_ip_i(t_i),\quad i=1,\ldots,N,\quad \hat u_j(t)=D_j(t)H_j(t)p(t), \quad  j=1,\ldots,M.
\end{equation}
 Setting
\begin{multline*}
u=\hat u=(D_1H_1p_1(t_1),\dots,D_iH_ip_i(t_i),\dots,D_NH_Np_N(t_N),\\D_1(t)H_1(t)p(t),\dots,D_j(t)H_j(t)p(t),\dots,D_M(t)H_M(t)p(t))
\end{multline*}
 in (\ref{infc}) and (\ref{point1}) and denoting  $\hat{z}(t)=z(t;\hat {u}),$
we see that $\hat{z}(t)$ and $p(t)$ satisfy system \eqref{Cau140aavv'} -- \eqref{Cau140aavv};  the unique solvability of
this system follows from the fact that
 functional $I(u)$ has one minimum point $\hat u$.

Now let us establish that $\sigma=[l(p)]^{1/2}.$
Substituting expression (\ref{oollb1x}) to (\ref{N4llvvv}),
we obtain
\begin{multline}\label{N4llvvvb}
\sigma(\hat u,\hat c)=I(\hat u)=\sum_{i=1}^{N+1}\int_{t_{i-1}}^{t_{i}}(\tilde Q(t) \hat z_{i}(t),\hat z_{i}(t))_ndt
\\+\sum_{i=1}^N(H_ip_i(t_i),D_iH_ip_i(t_i))_m+\sum_{j=1}^M\int_{\Omega_j}(H_j(t)p(t),D_j(t)H_j(t)p(t))_ldt.
\end{multline}
But
$$
\sum_{i=1}^{N+1}\int_{t_{i-1}}^{t_{i}}(\tilde Q(t) \hat z_{i}(t),\hat z_{i}(t))_ndt
=\sum_{i=1}^{N+1}\int_{t_{i-1}}^{t_{i}}\Bigl(\frac{dp_i(t)}{dt}-A(t)p_i(t),\hat z_{i}(t)\Bigr)_ndt
$$
$$
=\sum_{i=1}^{N+1}\Bigl((p_i(t_i),\hat z_{i}(t_i))_n-(p_i(t_{i-1}),\hat z_{i}(t_{i-1}))_n \Bigr)
$$
$$
+\sum_{i=1}^{N+1}\int_{t_{i-1}}^{t_{i}}\Bigl(p_i(t),-\frac{d\hat z_{i}(t)}{dt}-A^*(t)\hat z_{i}(t)\Bigr)_ndt
$$
$$
=\sum_{i=1}^{N+1}\int_{t_{i-1}}^{t_{i}}(p_i(t),l_0(t)-\sum_{i=1}^{N}(p_i(t_i),H_i^*D_iH_ip_i(t_i))_n-\sum_{j=1}^M \chi_{\Omega_j}(t)H^*_j(t)D_j(t)H_j(t)p(t))_ndt
$$
\begin{equation}\label{N4llvvvbb}
=l(p)-\sum_{i=1}^{N}(H_ip_i(t_i),D_iH_ip_i(t_i))_m-\sum_{j=1}^M\int_{\Omega_j}(H_j(t)p(t),D_j(t)H_j(t)p(t))_ldt
\end{equation}
The representation \eqref{ddd8vvv} follows from \eqref{N4llvvvb} and \eqref{N4llvvvbb}.

Prove that
\begin{equation}\label{Nllvvvbbb}
\widehat{\widehat{l(x)}}=l(\hat x).
\end{equation}
We should note, first of all that unique solvability of problem \eqref{Cau160avv} -- \eqref{Cau190aavv}  at realizations
$y_i,$ $i=1,\ldots,N,$ that belong with probability 1 to the space $\mathbb R^m$ can be proved similarly as to the problem
\eqref{Cau140aavv'} -- \eqref{Cau140aavv}.

Denote by $\hat p_i(t)$ and $\hat x_i(t)$ restrictions of $\hat p(t)$ and $\hat x(t)$, respectively, to $(t_{i-1},t_{i}),$ $i=1,\dots,N+1,$ extended by continuity to their ends.
Using \eqref{d6llvvves}, \eqref{ddd7llvvv}, and \eqref{Cau170avv}, we have
$$
\widehat{\widehat{l(x)}}=\sum_{i=1}^N(y_i,\hat u_i
)_m+\sum_{j=1}^M \int_{\Omega_j}(y_j(t),\hat u_j(t))_ldt+\hat c
$$
$$
=\sum_{i=1}^N(y_i,D_iH_ip_i(t_i))_m+\sum_{j=1}^M \int_{\Omega_j}(y_j(t),D_j(t)H_j(t)p(t))_ldt+\hat c
$$
$$
=\sum_{i=1}^N(H_i^*D_iy_i,p_i(t_i))_n+\int_0^T\sum_{j=1}^M \chi_{\Omega_j}(t)H_j^*(t)D_j(t)y_j(t),p(t))_ldt+\hat c
$$
$$
=\sum_{i=1}^N\Bigl(\hat p_{i}(t_i)-\hat p_i(t_{i+1})+H_i^*D_iH_i\hat x_i(t_i),p_i(t_i)\Bigr)_n
+\int_{0}^{T}\Bigl(B^*(t)\hat z(t),f_0(t)\Bigr)_{r}dt
$$
\begin{equation}\label{N4llvvvbbb}
+\int_0^T\Bigl(-\frac{d\hat p_i(t)}{dt}-A^*(t)\hat p_i(t)+\sum_{j=1}^M \chi_{\Omega_j}(t)H^*_j(t)D_j(t)H_j(t)\hat x(t),p_i(t)\Bigr)_ndt
\end{equation}
From \eqref{Cau140aavv'} -- \eqref{Cau140aavv} and \eqref{Cau160avv} -- \eqref{Cau190aavv}, we obtain
$$
\int_0^T\Bigl(-\frac{d\hat p_i(t)}{dt}-A^*(t)\hat p_i(t),p(t)\Bigr)_ldt=\sum_{i=1}^{N+1}\int_{t_{i-1}}^{t_i}\Bigl(-\frac{d\hat p_{i}(t)}{dt}-A^*(t)\hat p_{i}(t),p_{i}(t)\Bigr)_ndt
$$
$$
=\sum_{i=1}^{N+1}\Bigl((\hat p_{i}(t_{i-1}),p_{i}(t_{i-1}))_n-(\hat p_{i}(t_{i}),p_{i}(t_{i}))_n\Bigr)+\sum_{i=1}^{N+1}
\int_{t_{i-1}}^{t_i}\Bigl(\hat p_{i}(t),\frac{dp_{i}(t)}{dt}-A(t) p_{i}(t)\Bigr)_ndt
$$
$$
=(\hat p_1(t_0),p_1(t_0))_n+\sum_{i=2}^{N+1}(\hat p_{i}(t_{i-1}),p_{i}(t_{i-1}))_n-\sum_{i=1}^{N}(\hat p_{i}(t_{i}),p_{i}(t_{i}))_n-(\hat p_{N+1}(t_{N+1}),p_{N+1}(t_{N+1}))_n
$$
$$
+\sum_{i=1}^{N+1}\int_{t_{i-1}}^{t_i}\Bigl(\hat p_{i}(t),\tilde Q(t)\hat z_i(t)\Bigr)_ndt
$$
$$
=\sum_{i=1}^{N}(\hat p_{i+1}(t_{i}),p_{i}(t_{i}))_n-\sum_{i=1}^{N}(\hat p_{i}(t_{i}),p_{i}(t_{i}))_n
+\sum_{i=1}^{N+1}\int_{t_{i-1}}^{t_i}\Bigl(\tilde Q(t)\hat p_{i}(t),\hat z_i(t)\Bigr)_ndt
$$
$$
=\sum_{i=1}^{N}(\hat p_{i+1}(t_{i})-\hat p_{i}(t_{i}),p_{i}(t_{i}))_n
$$
\begin{equation}\label{N4llvvvbbb'}
+\sum_{i=1}^{N+1}
\int_{t_{i-1}}^{t_i}\Bigl(\frac{d\hat x_{i}(t)}{dt}-A(t) \hat x_{i}(t),\hat z_{i}(t)\Bigr)_ndt
-\sum_{i=1}^{N+1}\int_{t_{i-1}}^{t_i}\Bigl(B(t)f_0(t),\hat z_{i}(t)\Bigr)_ndt.
\end{equation}
But
$$
\sum_{i=1}^{N+1}
\int_{t_{i-1}}^{t_i}\Bigl(\frac{d\hat x_{i}(t)}{dt}-A(t) \hat x_{i}(t),\hat z_{i}(t)\Bigr)_ndt
$$
$$
=\sum_{i=1}^{N+1}\Bigl((\hat x_{i}(t_{i}),\hat z_{i}(t_{i}))_n-(\hat x_{i}(t_{i-1}),\hat z_{i}(t_{i-1}))_n\Bigr)+\sum_{i=1}^{N+1}
\int_{t_{i-1}}^{t_i}\Bigl(\hat x_{i}(t),-\frac{d\hat z_{i}(t)}{dt}-A^*(t) \hat z_{i}(t)\Bigr)_ndt
$$
$$
=\sum_{i=1}^{N}(\hat x_i(t_i),\hat z_i(t_i))_n
-\sum_{i=2}^{N+1}(\hat x_{i}(t_{i-1}),\hat z_{i}(t_{i-1}))_n+\int_{t_{0}}^{T}(\hat x(t),l_0(t)-\sum_{j=1}^M \chi_{\Omega_j}(t)H^*_j(t)D_j(t)H_j(t)p(t))_ndt
$$
$$
=l(\hat x)
+\sum_{i=1}^{N}(\hat x_i(t_i),\hat z_i(t_i)-\hat z_{i+1}(t_{i}))_n-\sum_{j=1}^M \int_{\Omega_j}(\hat x(t),H^*_j(t)D_j(t)H_j(t)p(t))_ndt
$$
\begin{equation}\label{N4llvvvbbb''}
=l(\hat x)
-\sum_{i=1}^{N}(\hat x_i(t_i),H^*_iD_iH_ip(t_i)))_n-\sum_{j=1}^M \int_{\Omega_j}(\hat x(t),H^*_j(t)D_j(t)H_j(t)p(t))_ndt.
\end{equation}
The representation \eqref{Nllvvvbbb} follows from \eqref{N4llvvvbbb}--\eqref{N4llvvvbbb''}.
\end{proof}
If observations are only pointwise, i.e.,
$
H_j(t)=0$ and $\xi_j(t)=0,$ $j=1,\dots,M,
$
in \eqref{sk7s},
the systems that generate minimax estimates
take the form
\begin{equation}\label{Cau140aavv'm}
-\frac{d\hat z(t)}{dt}=A^*(t)\hat z(t)+l_0(t),\quad t\in (0,T),\quad t\neq t_i,
\end{equation}
\begin{equation}\label{Cau160avv'm}
\Delta\hat z\left.\right|_{t=t_i}=H^*_iD_iH_ip(t_i),\quad i=1,\ldots,N,\quad \hat z(T)=\hat z(0),
\end{equation}
\begin{equation}\label{Cau170avv'm}
\frac{dp(t)}{dt}=A(t)p(t)+\tilde Q(t)\hat z(t),\quad t\in (0,T),\quad t\neq t_i,
\end{equation}
\begin{equation}\label{Cau140aavvm}
\Delta p\left.\right|_{t=t_i}=0,\quad i=1,\ldots,N,\quad p(0)=p(T)
\end{equation}
and
\begin{equation}\label{Cau160avvm}
-\frac{d\hat p(t)}{dt}=A^*(t)\hat p(t),\quad t\in (0,T),\quad t\neq t_i,
\end{equation}
\begin{equation}\label{Cau170avvm}
\Delta\hat p\left.\right|_{t=t_i}=H^*_iD_i[H_i\hat x(t_i)-y_i],\quad i=1,\ldots,N,\quad \hat p(T)=\hat p(0),
\end{equation}
\begin{equation}\label{Cau180aavvm}
\frac{d\hat x(t)}{dt}=A(t)\hat x(t)+\tilde Q(t)\hat p(t)+B(t)f_0(t),\quad t\in (0,T),\quad t\neq t_i,
\end{equation}
\begin{equation}\label{Cau190aavvm}
\Delta\hat x\left.\right|_{t=t_i}=0,\quad i=1,\ldots,N,\quad\hat x(0)=\hat x(T),
\end{equation}
respectively.
Show that in the case
the determination of functions $p(t)$ and $\hat z(t)$ from \eqref{Cau140aavv'} -- \eqref{Cau140aavv} can be reduced to solving
some system of linear algebraic equations.  To do this,
let us represent function $\hat z(t)$ in the form
\begin{equation}\label{hbe}
\hat z(t)=z(t;\hat u)=\bar z^{(0)}(t)+\sum_{i=1}^N\bar z^{(i)}(t;\hat u),
\end{equation}
where functions $\bar z^{(0)}(t)$ and $\bar z^{(i)}(t;u),$ $i=1,\dots,N,$ solve the problems
\begin{equation*}
-\frac{d\bar z^{(0)}(t)}{dt}=A^*(t)\bar z^{(0)}(t)+l_0(t),\quad t\in (0,T),\quad \bar z^{(0)}(T)=\bar z^{(0)}(0),
\end{equation*}
and
\begin{equation*}
-\frac{d\bar z^{(i)}(t;\hat u)}{dt}=A^*(t)\bar z^{(i)}(t;\hat u),\quad t\in (0,T),\quad t\neq t_i,
\end{equation*}
\begin{equation*}
\quad \bar z^{(i)}(t_i+0;\hat u)-\bar z^{(i)}(t_i;\hat u)=H^*_i\hat u_i,
\quad \bar z^{(i)}(T;\hat u)=\bar z^{(i)}(0;\hat u),
\end{equation*}
respectively.
It is easy to see that
\begin{equation}\label{p1}
\bar z^{(0)}(t)=Z(t)(E_n-Z(T))^{-1}\int_0^TZ(T)Z^{-1}(s)l(s)\,ds+\int_0^tZ(t)Z^{-1}(s)l(s),
\end{equation}
\begin{equation}\label{p2}
\bar z^{(i)}(t;\hat u)=Z(t)M_i(t)Z^{-1}(t_i)H^*_i\hat u_i,
\end{equation}
where $Z(t)=[X^*(t)]^{-1},$ $M_i(t)=(E_n-Z(T))^{-1}Z(T)+\chi_{(t_i,T)}(t)E_n,$ $\chi_{(t_i,T)}(t)$ is the characteristic function of the interval
$(t_i,T).$

In fact, \eqref{p1} is the direct corollary of the Cauchy formula. For proof of \eqref{p2}, we note that if $0<t<t_i$ then
\begin{equation}\label{p6}
\bar z^{(i)}(t;\hat u)=Z(t)\bar z^{(i)}(0;\hat u)
\end{equation}
and
$
\bar z^{(i)}(t_i-0;\hat u)=Z(t_i)\bar z^{(i)}(0;\hat u).
$
If $t_i<t<T$ then
$$
\bar z^{(i)}(t;\hat u)=Z(t)Z^{-1}(t_i)\bar z^{(i)}(t_i+0;\hat u)=
Z(t)Z^{-1}(t_i)(\bar z^{(i)}(t_i-0;\hat u)+H^*_i\hat u_i)
$$
$$
=Z(t)Z^{-1}(t_i)Z(t_i)\bar z^{(i)}(0;\hat u)+Z(t)Z^{-1}(t_i)H^*_i\hat u_i
$$
\begin{equation}\label{p7}
=Z(t)\bar z^{(i)}(0;\hat u)+Z(t)Z^{-1}(t_i)H^*_i\hat u_i.
\end{equation}
Setting in \eqref{p7} $t=T$ we find
$$
\bar z^{(i)}(T;\hat u)=Z(T)\bar z^{(i)}(0;\hat u)+Z(T)Z^{-1}(t_i)H^*_i\hat u_i
$$
Due to the equality $\bar z^{(i)}(T;\hat u)=\bar z^{(i)}(0;\hat u),$ it follows from here that
$$
\bar z^{(i)}(0;\hat u)=Z(T)\bar z^{(i)}(0;\hat u)+Z(T)Z^{-1}(t_i)H^*_i\hat u_i
$$
and
$$
\bar z^{(i)}(0;\hat u)=(E_n-Z(T))^{-1}Z(T)Z^{-1}(t_i)H^*_i\hat u_i.
$$
Substituting this expression into \eqref{p6} and \eqref{p7}, we obtain
\begin{equation}\label{p8}
\bar z^{(i)}(t;\hat u)=Z(t)(E_n-Z(T))^{-1}Z(T)Z^{-1}(t_i)H^*_i\hat u_i,\quad\mbox{if}\quad 0<t<t_i
\end{equation}
and
\begin{equation}\label{p9}
\bar z^{(i)}(t;\hat u)=Z(t)[(E_n-Z(T))^{-1}Z(T)+E_n]Z^{-1}(t_i)H^*_i\hat u_i,\quad\mbox{if}\quad t_i<t<T.
\end{equation}
Combining \eqref{p8}, and \eqref{p9} we get \eqref{p2}.

Further, using the Cauchy formula, \eqref{hbe}, \eqref{p1}, and \eqref{p2} we obtain from equations \eqref{Cau170avv'm} and \eqref{Cau140aavvm} that
\begin{equation}\label{p5}
p(t)=X(t)(E-X(T))^{-1}\int_0^TX(T)X^{-1}(s)\tilde Q(s)z(s)\,ds+\int_0^tX(t)X^{-1}(s)l(s)\tilde Q(s)z(s)\,ds
\end{equation}
$$
=X(t)(E-X(T))^{-1}\int_0^TX(T)X^{-1}(s)\tilde Q(s)\bar z^{(0)}(s)\,ds+\int_0^tX(t)X^{-1}(s)l(s)\tilde Q(s)\bar z^{(0)}(s)\,ds
$$
$$
+X(t)(E-X(T))^{-1}X(T)\sum_{k=1}^N\int_0^T X^{-1}(s)\tilde Q(s)Z(s)M_k(s)\,dsZ^{-1}(t_k)H^*_k\hat u_k
$$
$$
+X(t)\sum_{k=1}^N\int_0^t X^{-1}(s)\tilde Q(s)Z(s)M_k(s)\,dsZ^{-1}(t_k)H^*_k\hat u_k
$$
\begin{equation}\label{p3}
=X(t)C_0(t)+X(t)\sum_{k=1}^N[(E-X(T))^{-1}X(T)C_k(T)+C_k(t)]Z^{-1}(t_k)H^*_k\hat u_k,
\end{equation}
where
$$
C_0(t)=(E-X(T))^{-1}\int_0^TX(T)X^{-1}(s)\tilde Q(s)\bar z^{(0)}(s)\,ds+\int_0^tX^{-1}(s)l(s)\tilde Q(s)\bar z^{(0)}(s)\,ds,
$$
$$
C_k(t)=\int_0^t X^{-1}(s)\tilde Q(s)Z(s)M_k(s)\,ds.
$$
Setting in \eqref{p3} $t=t_i$ and $\hat{u}_k=D_kH_kp(t_k),$ $i=1,\dots,N,$ $k=1,\dots,N,$ we arrive at the following system of linear algebraic equations
for determination of unknown quantities $p(t_i):$
$$
p(t_i)=X(t_i)C_0(t_i)+X(t_i)\sum_{k=1}^N[(E-X(T))^{-1}X(T)C_k(T)+C_k(t_i)]Z^{-1}(t_k)H^*_kD_kH_kp(t_k)
$$
or
\begin{equation}\label{p4}
p(t_i)+\sum_{k=1}^N\alpha_{ik}p(t_k)=b_i, \quad i=1,\dots,N,
\end{equation}
where
$$
\alpha_{ik}=-X(t_i)\sum_{k=1}^N[(E-X(T))^{-1}X(T)C_k(T)+C_k(t_i)]Z^{-1}(t_k)H^*_kD_kH_k,\quad b_i=X(t_i)C_0(t_i).
$$
Finding $p(t_i)$ from \eqref{p4} we determine $\hat u_i,$ $\bar z^{(i)}(t;\hat u),$ $i=1,\dots,N,$ $\bar z^{(0)}(t),$ $\hat z(t),$
$p(t),$ and $c$ according to \eqref{ddd7llvvv},
 \eqref{p2}, \eqref{p1}, \eqref{hbe}, and \eqref{p5}, respectively.

In a similar way we can deduce a system of linear algebraic equations
via whose solution the functions $\hat x(t)$ and $\hat p(t)$ satisfying \eqref{Cau160avv} -- \eqref{Cau190aavv} are
expressed.

\renewcommand{\refname}
{\Large \bf
References
}

\end{document}